\documentclass[11pt]{amsart}
\usepackage[latin1]{inputenc}
\usepackage{amsmath,amsthm,amsfonts,amscd,amssymb,eucal,latexsym,mathrsfs}
\usepackage[all]{xy}
\usepackage{latexsym, amssymb, amscd, mathrsfs, graphics, fullpage}
\theoremstyle{definition}
\newtheorem{theorem}{Theorem}
\newtheorem{corollary}[theorem]{Corollary}
\newtheorem{lemma}[theorem]{Lemma}
\newtheorem{proposition}[theorem]{Proposition}
\newtheorem{remark}[theorem]{Remark}

\newtheorem{example}[theorem]{Example}

\newtheorem{definition}[theorem]{Definition}

\DeclareMathOperator{\dom}{dom}

\renewcommand{\Im}{\mathop{\mathrm{Im}}}

\def\max{\mathop{\mathrm{max}}\nolimits}

\newcommand{\CI}{{\mathrm{\bf C}}}\newcommand{\RI}{{\mathrm{\bf R}}}
\newcommand{\ZI}{{\mathrm{\bf Z}}}\newcommand{\NI}{{\mathrm{\bf N}}}

\newcommand{\degr}{{\rm deg}}

\newcommand{\R}{R}   
 \newcommand{\C}{C}
\newcommand{\G}{\Gamma}

\newcommand{\IIi}{\mathrm{II_1}}

\newcommand{\GL}{\mathrm {GL}}
\newcommand{\SL}{\mathrm {SL}}

\newcommand{\bE}{\mathbf{E}}

\newcommand{\norm}[1]{\|#1\|}
\newcommand{\abs}[1]{|#1|}

\newcommand{\tr}{\mathop{\mathrm{Tr}}}

\newcommand{\del}{\partial}

\newcommand{\h}{{\mathfrak{h}}}

\title[Uniform non-amenability, cost, and the first $\ell^2$-Betti
number]{Uniform non-amenability, cost, and the first $\ell^2$-Betti number}

\author{Russell Lyons$^*$}
\address{\hskip-\parindent Russell Lyons, Department of Mathematics,
Indiana University,
Bloomington, IN 47405-5701 }
\email{rdlyons@indiana.edu}

\author{Mika\"el Pichot}
\address{\hskip-\parindent
Mika\"el Pichot, I.H.E.S., 35 Route de Chartres, F-91440 Bures-sur-Yvette, France, and Isaac Newton Institute for Mathematical Sciences, Cambridge, U.K.}
\email{pichot@ihes.fr}

\author{St\'ephane Vassout}
\address{\hskip-\parindent St\'ephane Vassout, Institut de Math\'ematiques de Jussieu and Universit\'e Paris 7}
\email{vassout@math.jussieu.fr}

\thanks{
$^*$Supported partially by NSF grant DMS-0705518 and Microsoft Research.
}

\date{31 March 2008}

\begin{document}

\begin{abstract}
It is shown that  $2\beta_1(\G)\leq h(\G)$ for any countable group $\G$,
where $\beta_1(\G)$ is the first $\ell^2$-Betti number  and $h(\G)$ the
uniform isoperimetric constant.  In particular, a countable group with non-vanishing first $\ell^2$-Betti number is uniformly non-amenable.

We then define isoperimetric constants in the framework of measured
equivalence relations. For an ergodic measured equivalence relation $R$ of type
$\IIi$, the uniform isoperimetric constant $h(R)$ of $R$ is invariant under
orbit equivalence and satisfies 
$$
2\beta_1(R)\leq 2C(R)-2\leq h(R)
\,,
$$
where
$\beta_1(\R)$ is the first $\ell^2$-Betti number and $C(R)$ the cost of $R$
in the sense of Levitt (in particular  $h(R)$ is a non-trivial invariant). 
In contrast with the group case,  uniformly non-amenable measured equivalence relations of type $\IIi$ always contain non-amenable subtreeings.

An ergodic version $h_e(\G)$ of the uniform isoperimetric constant $h(\G)$
is defined as the infimum over  all essentially free ergodic and measure
preserving actions $\alpha$ of $\G$ of the uniform isoperimetric constant
$h(\R_\alpha)$ of the  equivalence relation $R_\alpha$ associated to
$\alpha$. By establishing a connection with the cost of measure-preserving
equivalence relations, we prove that $h_e(\G)=0$ for any lattice $\G$  in a
semi-simple Lie group of real rank at least 2 (while $h_e(\G)$ does not vanish in general).
\end{abstract}

\maketitle

\section{Introduction}

The isoperimetric constant of a graph offers a simple way to capture the isoperimetric behavior of finite sets in this graph. For a finitely generated countable group it reflects  isoperimetry in  Cayley graphs associated to finite generating sets of this group and is related to amenability.
In the present paper, we introduce an analogue of this constant for
measured equivalence relations of type $\IIi$. As we shall see, it behaves
in a very different manner from a measured dynamic point of view from the
uniform isoperimetric constant of a group.

Let $R$ be a measured equivalence relation of type $\IIi$ on a standard
(non-atomic) probability space $(X,\mu)$. Our main interest lies in two geometric
invariants that have recently been attached to $R$: the cost and the
$\ell^2$-Betti numbers. The cost of $R$  is a real number with values in
$[1,\infty]$ (assuming $R$ to be ergodic) denoted by $C(R)$.
From its definition one can readily infer that it is invariant under
isomorphism---i.e., orbit equivalence---of $R$ and the main problem is to
compute
it. In \cite{Gaboriau99} Damien Gaboriau established an explicit formula
relating the cost of an amalgamated free product (over amenable equivalence
subrelations) to the costs of their components; this allowed him to solve
the long-standing problem of distinguishing the free groups up to orbit
equivalence. In \cite{Gaboriau02}  he went further and introduced the
so-called $\ell^2$-Betti numbers of $R$. These are non-negative numbers
$\beta_0(R),\beta_1(R),\beta_2(R),\ldots$  in $[0,\infty]$ defined by
geometric constructs using an approximation process (as for their analogues for
countable groups; see   Cheeger and Gromov  \cite{CG86}) and one of the
main problems here (solved in \cite{Gaboriau02}) was to show that the
resulting numbers only depend on the isomorphism class of $R$. The first
$\ell^2$-Betti number provides another way to distinguish the free groups up to
orbit equivalence. 
The relation between $\ell^2$-Betti numbers and the cost is still unclear,
but the inequality $C(R)\geq \beta_1(R)+1$ is known to hold for any ergodic
equivalence relation of type $\IIi$ \cite{Gaboriau02}.  %
Recall that
$\ell^2$-Betti numbers were first introduced by Atiyah \cite{Atiyah76} in 1976 in his work on the index of equivariant elliptic operators on coverings spaces of Riemannian manifolds.
In the present paper we consider a new isomorphism invariant for $R$, the
uniform isoperimetric constant, $h(R)$. It takes values in $[0,\infty]$ and
is defined as an infimal value of isoperimetric ratios for `finite sets' in
the Cayley graphs of $R$ (see Section \ref{CheegerConstant}). For compact
Riemannian manifolds (and their coverings), the isoperimetric constant  was
considered by Cheeger  when he proved his well-known `Cheeger inequality'
relating it to the bottom of the spectrum of the Laplacian.

Our main theorem asserts that for any ergodic measured equivalence relation
of type $\IIi$, one has
\[
\beta_1(R)\leq C(R)-1\leq \frac{h(R)}{2}.
\]

Here the relation $R$ is assumed to be finitely generated, in which case
the uniform isoperimetric constant $h(R)$ has a finite value (a natural extension of
$h(R)$ to infinitely generated $R$ is to set $h(R)=\sup_{R'} h(R')$, where $R'$ runs over all finitely generated 
subrelations of $R$, 
so that, for instance, a measured equivalence relation $R$ given by a
measure-preserving and essentially free action of the free group $F_\infty$
on infinitely many generators will satisfy $\beta_1(R)=h(R)=\beta_1(F_\infty)=h(F_\infty)=+\infty$ \cite{Gaboriau02}).

We start by proving  the following weaker inequality in Section \ref{Cheeger}, 
\[
\beta_1(\G)\leq h(\G)
\]  
for a finitely generated group $\G$ (in particular, a countable group with
non-vanishing first $\ell^2$-Betti number is uniformly non-amenable). The
proof  uses only group-theoretic tools.  It  is inspired by  the proof of
Cheeger-Gromov's celebrated vanishing theorem in \cite{CG86} asserting that
when $\G$ is amenable, the sequence  
 \[
 \beta_0(\G),\beta_1(\G),\beta_2(\G),\ldots
 \] 
 of all  $\ell^2$-Betti numbers vanishes identically. Note that for an
 amenable  $\G$, one has $h(\G)=0$, but there do exist non-amenable groups with $h(\G)=0$ (see \cite{Osin,Arj}). 
 
Now the inequality $\beta_1(\G)\leq h(\G)$ is not optimal in general, but
we do not know how to get the optimal inequality using only group-theoretic arguments.
In Section \ref{perco}, we prove that in fact  
 \[
 \beta_1(\G)\leq \frac{h(\G)}{2}.
 \] 
Our proof relies on invariant
percolation theory. This inequality is  an equality for free groups, where
$2\beta_1(F_k)=h(F_k)=2k-2$, and thus is optimal. However, it can be
strict as there are groups of cost 1 that have $h>0$, for instance, higher rank lattices in semi-simple Lie groups.

The general inequality  relating the cost and the isoperimetric constant of
a measured equivalence relation, as stated above, is proved in Section
\ref{CheegerConstant}.  Together with Gaboriau's results \cite{Gaboriau02},
it implies that  $2\beta_1(\G)\leq h(\G)$ as well for every finitely generated group $\G$.

We call an ergodic measured groupoid $G$ (of type $\IIi$) uniformly
non-amenable if its isoperimetric constant $h(G)$ is non-zero. The class of
uniformly non-amenable groups is quite large and has been studied recently
by Osin (see, e.g., \cite{Osin,Osin2}) and  Arzhantseva, Burillo, Lustig, Reeves, Short, Ventura  (\cite{Arj}). Breuillard and Gelander \cite{Breuillard} have shown that for an arbitrary field $K$, any non-amenable and finitely generated subgroup of $\GL_n(K)$ is uniformly non-amenable. 

The class of uniformly non-amenable measured equivalence relations turns
out to be ``much smaller'' than its corresponding group-theoretic analogue.
For instance, if an equivalence relation $R$ is the partition into the
orbits of an essentially free measure-preserving action of a (non-uniform)
lattice in a higher rank Lie group, we have $h(R)=0$ (Corollary
\ref{lattice}) and thus $R$ is not uniformly non-amenable. Note that $R$
has the property T of Kazhdan (cf.\ \cite{ts} and references) in that case. Also, equivalence relations
that are decomposable as a direct product of two infinite equivalence
subrelations have trivial uniform isoperimetric constant (Corollary
\ref{product}). These results are derived by establishing a relation
between the cost and the uniform isoperimetric constant and by appealing to
some of Gaboriau's results in \cite{Gaboriau99}. Namely, we show (in
Section \ref{ErgodicCheeger}) that an ergodic equivalence relation with
cost 1 has a vanishing isoperimetric constant. The proof is reminiscent of the (non) concentration of measure property for measured equivalence relations (see \cite{ef}) which is implemented here via the Rokhlin Lemma. We also show that uniformly non-amenable measured equivalence relations have trivial fundamental groups (see Section \ref{Fund}; this will alternatively follow from \cite{Gaboriau99}) and always contain a non-amenable subtreeing (see Section \ref{vN}). 
The latter is related to the measure-theoretic analogue of the Day--von
Neumann problem. Recall that in the group case this problem (i.e., is it
true that every non-amenable group contains a non-amenable free group?) is
well known to have a negative answer, as was proved by Ol'shanskii. Osin
\cite{Osin2} showed that the answer was negative  as well even within the
class of uniformly non-amenable groups. To prove that it is positive in our situation, we combine our results with the corresponding known fact for equivalence relations with cost greater than 1 (\cite{KM,th}).

\bigskip

\noindent{\it Acknowledgements.} The second author was supported by an EPDI Post-doctoral Fellowship 
and is grateful to IHES for its  hospitality.

\section{The uniform non-amenability of groups with non-vanishing $\beta_1$}\label{Cheeger}
In this section, we recall the definitions of the uniform isoperimetric constant $h(\Gamma)$ and the first $\ell^2$-Betti number $\beta_1(\Gamma)$ for a countable group $\Gamma$ and show that $\beta_1(\Gamma) \leq h(\Gamma)$. This inequality is not optimal and will be improved in the next sections to $2\beta_1(\Gamma) \leq h(\Gamma)$ by using ergodic-theoretic arguments.

Let  $Y$ be a locally finite  graph and  $A$ be a subset of vertices of
$Y$. Define the  \emph{(edge) boundary  of $A$ in $Y$} to be the  set $\del_YA$ of edges of $Y$ with one extremity in $A$ and the other one in $Y\backslash A$. 
The  \emph{isoperimetric constant} of $Y$  is the non-negative number 
\[
h(Y) = \inf_{A\subset Y} \frac{\abs{ \del_Y A}}{\abs A},
\]
where the infimum runs over all finite subsets $A$ of vertices of $Y$.

Let $\G$ be a finitely generated group and $S$ be a finite generating set
of $\G$. Recall that the \emph{Cayley graph}  of $\G$ with respect to $S$
is the graph $Y$ whose vertices are the elements of $\G$ and whose edges
are given by right multiplication by elements of $S$. The isoperimetric
constant of this graph is called the \emph{isoperimetric constant of $\G$ with respect to $S$} and is denoted by $h_S(\G)$.   F\o lner's theorem asserts that a finitely generated group $\G$ is amenable if and only if $h_S(\G)=0$ for some (hence every) finite generating set $S$ \cite[Chap. VII]{Harpe00}.

By definition the  \emph{uniform isoperimetric constant} of  $\G$ is the infimum
\[
h(\G)=\inf_S h_S(\G)
\]
over all finite generating sets $S$ of $\G$.

\begin{example}
Osin \cite{Osin} has given examples of non-amenable groups with $h(\G)=0$.
For instance, he proved (see \cite[Example 2.2]{Osin}) that  the Baumslag-Solitar groups, with presentation 
\[
\mathrm{BS}_{p,q}=\langle a,t\mid t^{-1}a^pt=a^q\rangle\,,
\]
where $p,q>1$ are relatively prime, have vanishing uniform isoperimetric
constant (note that the uniform isoperimetric constant considered in
\cite{Osin} is defined in terms of the regular representation of the given
group: compare Section 13 in the paper of Arzhantseva et al.\ \cite{Arj}). 
\end{example}

Finitely generated groups with $h(\G)>0$ are called \textit{uniformly
non-amenable} (see  \cite{Arj}).  In fact our definition differs slightly
from the one given in \cite{Arj}  due to a different choice for the
boundary of a finite subset of vertices in a graph (basically the present
paper deals with the ``edge boundary" while the definition in \cite{Arj} involves the ``internal boundary"; see Appendix \ref{Hyperb} for more details). 
The Baumslag-Solitar groups have vanishing uniform isoperimetric constant for any reasonable  definition of the boundary.

The $\ell^2$-Betti numbers of a countable group $\G$ are non-negative real
numbers $\beta_0(\G),$ $\beta_1(\G), \ldots$ coming from
$\ell^2$-(co-)homology as $\G$-dimension (also known as Murray-von Neumann
dimension). We  refer to \cite{Gromov93,CG86,Pansu96,Luck02} for their
precise definition.  By a well-known theorem due to  Cheeger and Gromov
\cite{CG86}, the $\ell^2$-Betti numbers of a countable amenable group vanish
identically.  By elaborating on  ideas of \cite{CG86}  (see in particular
\textsection 3 in \cite{CG86}), in the case of non-amenable groups we obtain
the following explicit relation between the first $\ell^2$-Betti number and
the uniform isoperimetric constant.

\begin{theorem}\label{CG-groups}
Let $\G$ be a finitely generated group. Then $\beta_1(\G)\leq h(\G)$.
\end{theorem}

\begin{proof}
Let $S$ be a finite  generating set of $\G$ and $Y$ be the Cayley graph of $\G$ with respect to $S$. Write $C_i^{(2)}(Y)$, $i=0,1$,  for the space of square integrable  functions on the $i$-cells (vertices and edges) of $Y$. Associated to the simplicial boundary $\del_Y$ on $Y$ we have a bounded  operator 
\[
\del_1^{(2)} : C_1^{(2)}(Y)\to C_0^{(2)}(Y).
\]

Denote by $Z_1(Y)$ the space of finite 1-cycles and by $Z_1^{(2)}(Y)$ square integrable 1-cycles on $Y$.    Thus $Z_1^{(2)}(Y)$ is the kernel of $\del_1^{(2)}$ while $Z_1(Y)$ is space of functions with finite support in this kernel. The first $\ell^2$-Betti number $\beta_1(\G)$ of $\G$ coincides with  the Murray-von Neumann dimension 
\[
\beta_1(\G)=\dim_\G \bar H_1^{(2)}(Y)
\]
of   the orthogonal complement   $\bar H_1^{(2)}(Y) \subset C_1^{(2)}(Y)$  of the closed subspace $\overline {Z_1(Y)}$ in $Z_1^{(2)}(Y)$,
where the closure of $Z_1(Y)$ is taken with respect to the Hilbert norm. A
proof of this fact, together with the basic definitions used here, can be
found in \cite[section 3]{sc} (in particular, this step takes care of the approximation process involved in Cheeger-Gromov's definition of $\ell^2$-Betti numbers; compare to \cite{CG86}).

Let  $\Omega \subset Y^{(1)}$ be a finite subset of edges of $Y$ and consider the space 
\[
{\bar H_1^{(2)}(Y)}_{|\Omega} = \{ \sigma_{|\Omega} \,;\;  \sigma\in \bar H_1^{(2)}(Y)\}
\]
  of restrictions of harmonic chains to $\Omega$. This is a linear subspace of the space $C_1(\Omega)$ of complex functions on $\Omega$. Let 
 $$P : C_1^{(2)}(Y) \to C_1^{(2)}(Y)$$ 
  be the (equivariant) orthogonal projection on $\bar H_1^{(2)}(Y)$ and 
  $$ \chi_\Omega : \C_1^{(2)}(Y) \to C_1^{(2)}(Y)$$ 
 be the orthogonal projection on $C_1(\Omega)$.

Given a finite set $A$ of $\G$, we denote by $A_S$ the set of edges with a
vertex in  $A$. We have $\del_S A\subset A_S$, where $\del_S A$ is the  boundary of $A$ in $Y$ defined in Section \ref{Cheeger}. Let us now prove that, for every non-empty finite set $A$ of $\G$,
\[
\beta_1(\G) \leq  \frac 1 {\abs A} \dim_\CI   {\bar H_1^{(2)}(Y)}_{|A_S}.
\]
Write $M_{A_S}$ for the composition $\chi_{A_S}P$, considered as an operator from  $\C_1(A_S)$ to itself with range $\bar H_1^{(2)}(Y)_{|A_S}$. We have
\begin{align*}
\dim_\G  \bar H_1^{(2)}(Y) &= \sum_{s\in S} \langle P\delta_{(e,s)}\mid \delta_{(e,s)}\rangle 
=  \frac 1 {\abs A} \sum_{a \in A, s \in S} \langle P\delta_{(a,as)} \mid \delta_{(a,as)}\rangle \\
&\hspace{7mm}\leq \frac 1 {\abs A} \sum_{u\in A_S} \langle P\delta_u \mid \delta_u\rangle 
= \frac 1 {\abs A} \tr M_{A_S}\\
& \hspace{15mm}\leq \frac 1 {\abs A} \dim_\CI \bar H_1^{(2)}(Y)_{|A_S}
\end{align*}

where the last inequality follows from the fact that $\norm{M_{A_S}}\leq 1$. This gives the desired inequality.

We now observe that every harmonic 1-chain (i.e., element of $\bar
H_1^{(2)}(Y)$) that ``enters" a subset $A_S$ has to intersect its boundary $\del_S A$:

\begin{lemma}\label{restr}
Let $A$ be a finite subset of $\G$.  The canonical (restriction) map 
\[
r : \bar H_1^{(2)}(Y)_{|A_S} \to \bar H_1^{(2)}(Y)_{|\del_S A}
\]
is injective.
\end{lemma}

\begin{proof}
Recall that 
\[
Z_1^{(2)}(Y) = \bar H_1^{(2)}(Y) \oplus_\perp \overline {Z_1(Y)}.
\]  
Let $\sigma \in \bar H_1^{(2)}(Y)$. If $\sigma$ vanishes on $\del_S A$, then $\sigma_{|A_S}$ is  a finite 1-cycle as the boundary operator commutes with the restriction to $A_S$ in that case. Thus $\sigma_{|A_S}$ vanishes identically. 
\end{proof}

\noindent \textit{Back to the proof of Theorem \ref{CG-groups}}.   The above Lemma \ref{restr} gives
\[
\dim_\CI \bar H_1^{(2)}(Y)_{|A_S}=\dim_\CI \bar H_1^{(2)}(Y)_{|\del_S A}
 \leq \abs{\del_S A}\,,
\]
which immediately yields the theorem:
\[
\dim_\G  \bar H_1^{(2)}(Y)  \leq \frac {\abs{\del_S A}} {\abs A}
\]
and thus $\beta_1(\G)\leq h(\G)$.
\end{proof}

\begin{remark}
 L\"uck's generalization of the Cheeger-Gromov theorem to arbitrary module
 coefficients (Theorem 5.1 in \cite{Luck}) does not hold for groups with
 vanishing uniform isoperimetric constant (as these groups may contain non-abelian free groups; compare Remark 5.14 in \cite{Luck}).
\end{remark}

\begin{remark}
As mentioned above, the inequality  $\beta_1(\G)\leq h(\G)$ is not optimal
in general. Consider, for example, the case of the free group $F_k$ on $k$
generators. As is well known, one has  $\beta_1(F_k)=k-1$ and   the uniform
isoperimetric constant $h(F_k)=2k-2$ can be computed by considering large
balls with respect to a fixed generating set, \textit{e.g.}, the usual system $S_k$ of $k$ free generators (see Appendix \ref{bounds}).
\end{remark}

\section{The optimal inequality for countable groups {\em{via}} percolation theory}\label{perco}

Percolation theory is concerned with random subgraphs of a fixed graph,
usually a Cayley graph.
Recently, the notions of cost and $\ell^2$-Betti number, as well as the
theory of equivalence relations, have found use in percolation theory; see,
e.g., \cite[Remarks after Conj.~3.8]{Lyonsphase}, \cite{GabHD}, or
\cite[Cor.~3.24]{LPSmsf}.
Percolation theory has also been used in the theory of equivalence
relations: see \cite{GabLyons}.
Here, we use percolation theory to give a short proof
of the optimal inequality relating uniform isoperimetric constants to
$\ell^2$-Betti numbers.

We call a graph \emph{transitive} if its automorphism group acts
transitively on its vertices.
A graph is a \emph{forest} if it contains no cycles.
A \emph{percolation} on a graph $Y$ is a probability measure on
subgraphs of $Y$; 
it is a \emph{foresting} if it is concentrated on forests, while it is
\emph{invariant} if it is invariant under all automorphisms of $Y$.
We denote the degree of a vertex $x$ in $Y$ by $\degr_Y x$.
If $F$ is an invariant percolation on a transitive graph $Y$, then by
definition,
for every vertex $x \in Y$, we have $2C(F) = \bE[\degr_F x]$.
See Def.~2.10 of \cite{GabHD} for the definition of $\beta_1(Y)$ when $Y$ is
a general transitive graph.

\begin{theorem}\label{rsf}
If $Y$ is a transitive graph and $F$ is a random invariant foresting of $Y$,
then 
$$ C(F)-1 \leq \frac{h(Y)}{2}.$$
\end{theorem}
\begin{proof}
Let $A$ be a finite set of vertices in $Y$.
We have
$$
2C(F)
= \frac{1}{\abs A} \sum_{x \in A} \bE[\degr_F x]
= \frac{1}{\abs A} \bE\Big[\sum_{x \in A} \degr_F x\Big]
\le \frac{2 \abs A + \abs{\del_Y A}}{\abs A}
=
2 + \frac{\abs{\del_Y A}}{\abs A}
.
$$
Taking the infimum over $A$ gives the desired inequality.
\end{proof}

\begin{corollary}\label{fsf}
Let $Y$ be a transitive graph.
Then 
$$ \beta_1(Y) \leq \frac{h(Y)}{2}.$$
\end{corollary}
\begin{proof}
Apply Theorem \ref{rsf} to the free uniform spanning forest $F$ of $Y$; see
\cite{BLPSusf} for its definition.
The fact that $C(F) - 1 = \beta_1(Y)$ follows from Theorems 6.4 and 7.8 of
\cite{BLPSusf},
an identity first observed in \cite{Lyonsbetti} for Cayley
graphs.

\end{proof}

In particular, if $\G$ is a finitely generated group, then $\beta_1(\G) \le
h(\G)/2$.

\section{The case of measured equivalence relations}\label{CheegerConstant}
In this section, we recall the definition of the first $\ell^2$-Betti number
and the cost of a measured equivalence relation, define the isoperimetric
constant, and show that $2\beta_1(R)\leq 2C(R)-2 \leq h(R)$.

\subsection{Measured equivalence relations}
Let $(X,\mu)$ be a standard (non-atomic) probability space. An equivalence relation $R$
with countable classes on $X$ is called Borel if its graph $R\subset
X\times X$ is a Borel subset of $X\times X$. It is called a measured
equivalence relation if the $R$-saturation of a $\mu$-null subset of $X$ is
again $\mu$-null. 
For instance, if $\G$ is a countable group and $\alpha$ is a non-singular action of $\G$ on $(X,\mu)$, then the associated equivalence relation $R_\alpha$, defined as the partition of $X$ into $\G$-orbits, is a measured equivalence relation on $X$. See \cite{FM77} for more details.

 Let   $R$ be a measured  equivalence relation on  a standard probability
 space $(X,\mu)$. We assume throughout the paper that $R$ is
 \emph{ergodic}, i.e., that every saturated measurable subset of $X$ has measure 0 or 1 (in fact our results can be generalized to equivalence relations with \emph{infinite classes}). Endow $R$ with the horizontal counting measure $\h$ given by 
\[
\h(K)=\int_X\abs{K^x}d\mu(x)
\]
where $K$ is a measurable subset of $R$ and $K^x$ is the subset of $X\times X$ defined as $K^x=\{(x,y)\in K\}$. A partial automorphism  of $R$ is a partial automorphism of $(X,\mu)$ whose graph is included in $R$. One says that $R$ is of type $\IIi$ if the measure $\mu$ is invariant  under every partial automorphism of $R$.
The group of (full) automorphisms of $R$ is denoted $[R]$.

Recall that a \emph{graphing} of a measured equivalence relation can be
described in either one the following two ways (cf.\ \cite{Gaboriau99}):
\begin{enumerate}
\item a family $\Phi=\{\varphi_i\}_{i\geq 1}$ of partial automorphisms of $R$ such that for almost every $(x,y)\in R$, there exists a finite sequence  $(\varphi_1,\varphi_2,\ldots,\varphi_n)$ of elements of $\Phi\cup \Phi^{-1}$  such that $y=\varphi_n\ldots \varphi_1(x)$, 
\item a measurable subset $K$ of $R$  such that $R$ coincides with
$\bigcup_1^\infty K^n$ up to a negligible (i.e., $\mu$-null) set, where $K^n$ is the $n$-th convolution product of $K$.
\end{enumerate}

A graphing $K$ of $R$ is said to be \emph{finite} if it can be partitioned into a finite number of partial automorphisms of $R$. A type $\IIi$ equivalence relation $R$ on $(X,\mu)$ is said to be \emph{finitely generated} if it admits a finite graphing (this is equivalent to saying that  $R$ has finite cost; see \cite{Gaboriau99}).
\subsection{The first $\ell^2$-Betti number}
We sketch  the definition of the first $\ell^2$-Betti number for a
measure-preserving equivalence relation. For more details and proofs, see
the  paper of Gaboriau \cite{Gaboriau02} (in particular, Section 3.5).

Let $(X,\mu)$  be a standard probability space and  $\Sigma$ be a
measurable field of oriented 2-dimensional cellular complexes. For
$i=0,1,2$, we write $\CI[\Sigma^{(i)}]$ for the algebras of functions on the
$i$-cells of  $\Sigma$ that have uniformly finite support: a function $f :
\Sigma^{(i)}\to \CI$ is in $\CI[\Sigma^{(i)}]$ if and only if there exists
a constant $C_f$ such that for almost every $x\in X$, the number of
$i$-cells $\sigma$ of $\Sigma^x$ such that $f(\sigma)\neq 0$ is bounded by $C_f$.  The completion of $\CI[\Sigma^{(i)}]$ for the norm 
\[
\norm f_2^2= \int_X \sum_{\sigma\ i\text{-cells in }\Sigma^x}\abs{f(\sigma)}^2d\mu(x)
\] 
is a Hilbert space, which we  denote by $C_i^{(2)}(\Sigma)$. If $\Sigma$ is
uniformly locally finite, i.e., if the number of cells attached to a vertex $\sigma\in\Sigma^{(0)}$ is almost surely bounded by a constant $C$, then the natural (measurable fields of) boundary operators $\del_i : \CI[\Sigma^i]\to \CI[\Sigma^{i-1}]$ 
coming from the `attaching cells maps' extend to bounded operators
$\del_i^{(2)} : C_i^{(2)}(\Sigma)\to C_{i-1}^{(2)}(\Sigma)$. The
\textit{first reduced $\ell^2$-homology space} of $\Sigma$ is then the quotient space 
\[
\overline H_1^{(2)}(\Sigma)=\ker\del_1^{(2)}/\overline{\Im\del_2^{(2)}}
\] 
of the kernel of $\del_1^{(2)}$ by the (Hilbert) closure of the image of $\del_2^{(2)}$.  It is naturally isometric to the orthogonal complement   of $\overline{\Im\del_2^{(2)}}$  in $\ker\del_1^{(2)}$.

Now let $R$ be an ergodic equivalence relation of type $\IIi$ on $(X,\mu)$.
Consider a measurable field of oriented 2-dimensional cellular complexes
$\Sigma$  endowed with a measurable action of $R$  with fundamental domain
(see Section 2 in \cite{Gaboriau02}; a fundamental domain $D$ in $\Sigma$
is a measurable set of cells of $\Sigma$ intersecting almost each $R$-orbit
at a single cell of $D$).   Assume that $\Sigma$ is uniformly locally
finite and denote by $N$ the von Neumann algebra of $R$ (see, e.g., Section 1.5 in \cite{Gaboriau02}). The first $\ell^2$-homology space $\overline H_1^{(2)}(\Sigma)$ is then a Hilbert module and it has a Murray-von Neumann dimension over $N$. This dimension is called the first $\ell^2$-Betti number of $\Sigma$ and is denoted by $\beta_1(\Sigma,R)$.

Gaboriau \cite{Gaboriau02} has extended this definition beyond  the 
uniformly locally finite case by using an approximation
technique in the spirit of  Cheeger-Gromov \cite{CG86}.   He then proved
that \emph{the associated $\ell^2$-Betti number $\beta_1(\Sigma,R)$ is
independent  of the choice of $\Sigma$ provided that $\Sigma$ is simply
connected} (i.e., almost each fiber is simply connected). We refer to this
result as  Gaboriau's homotopy invariance theorem (it holds for all
$\ell^2$-Betti numbers; see  \cite[Th\'eor\`eme 3.13]{Gaboriau02}).  The
number $\beta_1(\Sigma, R)$ for a simply connected $\Sigma$  is called the
first $\ell^2$-Betti number of  $R$ and is denoted by $\beta_1(R)$ (note
that such a $\Sigma$ always exists---for instance, one can take the classifying space $ER$ of $R$; see Section 2.2.1 in \cite{Gaboriau02}).    

\subsection{Cost}
Let $R$ be an ergodic 
equivalence relation of type  $\IIi$ on a probability space $(X,\mu)$.  The cost of a partial automorphism $\varphi : A\to B$  of $R$ is the measure of its domain, $C(\varphi)=\mu(A)$. The cost of a graphing $\Phi=\{\varphi_i\}_{i\geq 1}$ of $R$ is defined to be
\[
C(\Phi)=\sum_{i\geq 1}C(\varphi_i)
\]
while the \emph{cost of $R$} is the infimum
\[
C(R)=\inf_{\Phi} C(\Phi)\,,
\]
where $\Phi$ runs among all graphings $\Phi$ of $R$. The \emph{cost of a countable group} $\G$ is the infimum 
\[
C(\G)=\inf_{\alpha} C(R_\alpha)
\]
where $\alpha$  runs over  all ergodic essentially free measure-preserving actions $\alpha$ of $\G$ on a probability space. This definition has been introduced by Levitt.
See \cite{KM} for an exposition.

\subsection{Isoperimetric constants}
Let $R$ be an ergodic equivalence relation of type  $\IIi$ on a probability
space $(X,\mu)$ and let $K$ be a graphing of $R$.  We now define the
isoperimetric constant $h_K(R)$ of $R$ with respect  to   $K$.  Consider
the measurable field of graphs $\Sigma=\amalg_{x\in X}\Sigma^x$ over $X$
defined as follows (see  \cite{Gaboriau02}, section 2):  the vertices of $\Sigma^x$
are the elements of $R^x$ and the set of edges of $\Sigma^x$ is the family
of pairs $\big((x, y), (x, z)\big) \in R^x \times R^x$ such that $(y, z)
\in K$.

There is an obvious action of $R$ on $\Sigma$ by permutation of fibers. In
concordance with  group theory, we shall call $\Sigma$ the \emph{Cayley graph} of $R$  associated to $K$. This is an example of a ``quasi-periodic metric space" associated to $R$ (\cite{th}). 

We write $\Sigma^{(0)}$ for the set of vertices  and  $\Sigma^{(1)}$ for the set  of edges of $\Sigma$ (thus $\Sigma^{(0)}=R$). We define vertices in $\Sigma$  as follows. This strengthens the corresponding definition of vertices in \cite{th} so as to fit our present purposes (in \cite{th} vertices of $\Sigma$ were defined to be  partially supported \emph{sections of vertices} of the map $\Sigma\to X$).

\begin{definition}\label{Points}
Let $\Sigma$ be the Cayley graph of $R$ with respect to a graphing $K$. By
a \emph{symmetric vertex of $\Sigma$} 
we mean the graph in $\Sigma^{(0)}$
of  an automorphism of the equivalence relation $R$. We shall say that two
symmetric vertices  of $\Sigma$ are  \emph{disjoint} if the corresponding graphs in
$\Sigma^{(0)}$ have $\h$-null intersection.
\end{definition}

Given a finite set $A$ of symmetric vertices of $\Sigma$, we denote by $\del_K A\subset \Sigma^{(1)}$ the  set of edges of $\Sigma$ with one vertex in $A$ and the other one outside of $A$. We endow $\Sigma^{(1)}$ with the measure $\nu^{(1)}$ defined by
\[
\nu^{(1)}(E)=\int_X \abs{(E\cap \Sigma^x)}d\mu(x)
\]
for a measurable subset $E$ of $\Sigma^{(1)}$.

\begin{definition} \label{Chcst} The \textit{isoperimetric constant} of $R$ with respect to a finite graphing $K$ is the non-negative number
\[
h_K(R)=\inf_{A} \frac{\nu^{(1)}(\del_KA) }{ \abs A},
\]
where the infimum is taken over all finite sets $A$ of  pairwise disjoint symmetric vertices of $\Sigma_K$. 
\end{definition}

The \textit{uniform isoperimetric constant} of a finitely generated equivalence relation $R$ is the non-negative number 
\[
h(R)=\inf_{K\subset R} h_K(R)
\]
where the infimum is taken over the finite graphings $K$ of $R$. Note that
by definition the uniform isoperimetric constant  of an equivalence relation is invariant under isomorphism.

\begin{remark}
For every graphing $K$ of $R$,
\[
h_K(R)\ge\int_X h(\Sigma_K^x) d\mu(x),
\] 
where $h(\Sigma_K^x)$ is the isoperimetric constant of the graph
$\Sigma_K^x$, but strict inequality may occur. The reason   is that measured equivalence relations 
 always admit graphings  having \emph{vanishing F\o lner sequences} \cite[D\'efinition 7.1]{ef} without this having any consequence 
 on their algebraic structure (compare \cite[Th\'eor\`eme 7.5]{ef}). Thus
 for any type $\IIi$ equivalence relation $R$, one gets by taking the graphing $K$ of Exemple 7.3 in \cite{ef}
 that   
\[
\int_X h(\Sigma_K^x) d\mu(x)=0
\] 
and a definition of $h(R)$ as the infimum over $K$ of $\int_X h(\Sigma_K^x)
d\mu(x)$ would lead to a trivial invariant. In fact, it is not clear either
that  our definition of $h(R)$ can achieve  non-trivial numbers.  It
involves an infimum over all possible (finite) graphings of $R$, in the
spirit of the cost, so that a proof that $h(R)$ is a non-trivial invariant
requires an homotopy-invariance type of argument. One can give a proof of that relying on Gaboriau's homotopy invariance theorem  \cite[Th\'eor\`eme 3.13]{Gaboriau02} for $L^2$ Betti numbers 
(by a straightforward adaptation of Theorem \ref{CG-groups} to equivalence relations using Definition \ref{Points}). The proof we present below (Theorem \ref{main}) relies on computations of the cost in \cite{Gaboriau99}.
\end{remark}

\begin{definition}
An ergodic equivalence relation of type $\IIi$ is called \emph{uniformly non-amenable} if $h(R)>0$.
\end{definition}

Note that by  Connes-Feldman-Weiss theorem \cite{CFW81} (amenable equivalence relations are singly generated),   a uniformly
non-amenable ergodic equivalence relation of type $\IIi$ is not amenable.

\subsection{The inequality for measured equivalence relations}

\begin{theorem}\label{main}
Let $R$ be a finitely generated ergodic  equivalence relation of type $\IIi$. Then
$$ \beta_1(R) \leq C(R)-1 \leq \frac{h(R)}{2}.$$
\end{theorem}
\begin{proof}
The first inequality is due to Gaboriau \cite[Corollaire 3.22]{Gaboriau02} so we concentrate
on the second one. As  was shown  by  the second-named author (see
\cite{th}) and (simultaneously) by Kechris and Miller  \cite{KM}, the
equivalence relation $R$ contains subtreeings of cost at least $C(R)$. More
precisely, by Corollaire 39, p 25 in  \cite{th} , for any given graphing $K$ of $R$, there exists a treeing
$F\subset K$ generating an equivalence subrelation $S$ of $R$ with
$C(F)=C(S)\geq C(R)$. Recall that a treeing is a graphing  with no simple cycle (see \cite[Definition I.2]{Gaboriau99}). 
To prove the  theorem, we thus are left to show that 
\[
2C(F) \leq 2+h(R).
\] 
Let $A$ be a finite set of pairwise disjoint symmetric vertices  in $\Sigma$ in the sense of  Definition \ref{Points}, so elements of $A$ are (graphs of) automorphisms of $R$. For $x\in X$, let $\Sigma_F^x$ denote the family of subtrees of $\Sigma_K^x$ associated to $F$. 
  Since  $C(F)=\frac{1}{2}\int_{x \in X} \degr_F(x)d \mu(x)$,  we have
\[
2C(F)=\frac{1}{\abs{A}}\sum_{\varphi \in A} \int_{x \in X} \degr_F(\varphi(x))d
\mu(x)=\frac{1}{\abs{A}} \int_{x\in X} \degr_{F}(A^x)d\mu(x),
\]
where $\degr_{F}(A^x)$ is the total degree of the points of $A^x$ in $\Sigma_F^x$.
Now as  $\Sigma_F^x$ is a family of trees, we get  
\[
\degr_{F}(A^x)= 2\abs{E^x_A}+\abs{\del_{F}A^x} \leq 2\abs{A^x}+\abs{\del_{F}A^x}
\leq 2\abs{A}+\abs{\del_{K}A^x},
\]
where $E_A^x$ is the set of the edges in $\Sigma_F^x$ with both vertices in $A^x$ and $\del_{F}A^x$ the  edges in $\Sigma_F^x$ with exactly one vertex in $A^x$.
Hence 
\[
2C(F)\leq 2+\frac{1}{\abs{A}} \int_{x\in X} \abs{\del_{K}A^x}d\mu(x)
=
2+\frac{\nu^{(1)}(\del_KA) }{ \abs A}
.
\]
Taking the infimum over $A$ and then over $K$, we get $2C(F)\leq 2+h(R)$, as desired.
\end{proof}

\begin{remark}
The proof of Section \ref{Cheeger} can also be adapted to give the inequality $\beta_1(R) \leq h(R)$, and in fact it gives the same inequality for any  $r$-discrete measured groupoid  \cite{AnanRen}.
\end{remark}

\subsection{Simultaneous vanishing of the invariants}
Theorem \ref{main} shows in particular that $h(R)=0$ implies $C(R)=1$. We now prove the converse.

\begin{proposition}\label{hzero}
Let $R$ be an ergodic equivalence relation with cost 1. Then $h(R)=0$. 
\end{proposition}

\begin{proof}
Let $R$ be an ergodic equivalence relation with cost 1. Let $\varphi$ be an
ergodic automorphism of $R$ and for each real number $\varepsilon >0$, let
$\psi_\varepsilon$ be a partial automorphism of $R$ of cost $\varepsilon$
such that $$\Phi_\varepsilon=(\varphi,\psi_\varepsilon)$$ is a graphing of
$R$. The existence of such graphings of $R$ is proved in \cite[Lemma III.5]{Gaboriau99}.
Denote by $R_\varphi$ the equivalence relation generated by $\varphi$ and
fix  $n\in \NI$. By Rokhlin's Lemma there exists a family
$B_1,\ldots,B_{n}$ of disjoint subsets of $X$ such that 
\[
\varphi(B_i)=B_{i+1}~~\text{for}~~i=1,\ldots, n-1~~\text{and}~~
\mu\big(X\backslash \bigcup_1^{n} B_i\big)\leq 1/4.
\]
Suppose that  $\varepsilon<\mu(B_1)$ and consider two partial automorphisms $\theta_{\varepsilon,1}$ and $\theta_{\varepsilon,2}$ of the equivalence relation $R_\varphi$ such that $\theta_{\varepsilon,1}(\dom \psi_\varepsilon)\subset B_1$ and $\theta_{\varepsilon,2}(\Im \psi_\varepsilon)\subset B_1$. Set 
\[
\psi'_\varepsilon=\theta_{\varepsilon,2}\psi_\varepsilon\theta_{\varepsilon,1}^{-1}.
\]
It is not hard to check that $\Phi_\varepsilon'=(\varphi,\psi_\varepsilon')$ is a graphing of $R$. Now letting $n\to \infty$ we obtain that
 for any integer $n\in \NI$ and any $\varepsilon>0$ there exists a graphing $\Phi_{\varepsilon,n}$ of $R$ of the form $\Phi_{\varepsilon,n}=(\varphi,\psi_{\varepsilon,n})$ whose  cost is less than $1+\varepsilon$  and such that for almost every  $x\in X$, the intersection of either the domain or the image of $\psi_{\varepsilon,n}$ with the finite set 
 \[
 \{x,\varphi(x),\ldots,\varphi^n(x)\}
 \]
 is at most one point 
  (this  should be compared to  the fact that the ``concentration of
  measure" property  fails for   automorphisms of standard probability
  spaces by Rokhlin's Lemma; see \cite{ef}). Let us now consider  the Cayley graph $\Sigma_{\varepsilon,n}$ associated to $\Phi_{\varepsilon,n}$ as in Section \ref{CheegerConstant}. Let $A_n$ be the set of (pairwise disjoint) symmetric points  of $\Sigma_{\varepsilon,n}$ given by $A_n=\{\varphi^i\}_{i=0}^n$.  Then the boundary of $A_n^x$ in $\Sigma_{\varepsilon,n}^x$ consists of  at most 4 points for almost every $x\in X$. 
 It follows that $h(R)=0$. 
 
\end{proof}

\section{Some consequences of the main inequality}
\subsection{On the Day-von Neumann problem for uniformly non-amenable equivalence relations}\label{vN}
The question of the existence of non-abelian free groups in non-amenable
groups is often referred to as the (Day-)von Neumann problem. It was solved
negatively by Ol'shanskii in 1980;  Adian (1982) showed that the free
Burnside groups with large (odd) exponent are non-amenable. Since they do
not contain free groups, this gave another negative solution to the Day-von
Neumann problem.
Osin \cite{Osin2}  extended Adian's result to show that these Burnside groups
are uniformly non-amenable (and even that the
regular representation has non-vanishing uniform Kazhdan's constant),
whence he deduced the existence of
finitely generated groups that are uniformly non-amenable and do not
contain any free group on two generators \cite[Theorem 1.3]{Osin2}. 

The Day-von Neumann problem is an open question for measured equivalence relations. It can be formulated in the following way. 
Let $(X,\mu)$ be a standard probability space and $R$ be a non-amenable
ergodic equivalence relation of type $\IIi$ on $(X,\mu)$. Is it true that
$R$ contains a non-amenable subtreeing?  The result of Kechris-Miller and
of the second author recalled in the proof of Theorem \ref{main} shows that
every ergodic equivalence relation of type $\IIi$ with cost greater than 1
(and thus non-amenable) contains a non-amenable subtreeing. Combining this
with Proposition \ref{hzero}, we get the following result.

\begin{corollary}[See \cite{Osin2} for the group case]
Let $R$ be a uniformly non-amenable ergodic equivalence relation of type $\IIi$. Then $R$ contains  a non-amenable subtreeing.
\end{corollary}

\subsection{Fundamental groups}\label{Fund}

Let $(X,\mu)$ be a probability space and $R$ be an ergodic equivalence relation of type $\IIi$ on $(X,\mu)$. The so-called \textit{fundamental group} of $R$ is the multiplicative subgroup of $\RI_+^*$ generated by the measure of all Borel subsets $Y$ of $X$ such that the restricted equivalence relation $R_{|Y}$ is isomorphic to $R$.

The next corollary follows from Proposition \ref{hzero} and the fact that equivalence relations with non-trivial cost have trivial fundamental group, which is proved in \cite{Gaboriau99}, Proposition II.6.  
\begin{corollary}
A finitely generated ergodic equivalence relation of type $\IIi$ that is uniformly non-amenable has a trivial fundamental group.  
\end{corollary}

This corollary can also be  proved directly by using the following compression formula.

\begin{proposition}
 Let $R$ be a finitely generated ergodic equivalence relation on $(X,\mu)$ and $Y\subset X$ be a non-negligible measurable subset of $X$. Let $S$ be the restriction of $R$ to $Y$. Then 
\[
h(R)\leq \mu(Y)h(S).
\]
\end{proposition}

\begin{proof}

Note that   (by definition)   $h(S)$ should be computed  with respect to
the normalized measure $\mu_1=\frac{\mu}{ \mu(Y)}$.   
Let us prove this inequality (we do not know whether it is an equality). 

Fix $\varepsilon\in (0,\mu(Y))$. Let $K$ be a (finite) graphing of $S$ such that $h_K(S)\leq h(S)+\varepsilon/4$ and let $A$ be a finite set of pairwise distinct vertices of  $\Sigma_K$ such that  
\[
 \frac{\nu_1^{(1)}(\del_KA) }{ \abs A}< h_K(S) +\varepsilon/4
\] 
and $\abs  A>12/\varepsilon$ (that one can choose $A$ arbitrary large follows by \emph{quasi-periodicity}; see the theorem of ``repetition of patterns" on p.~56 of \cite{th}), 
where $\nu_1^{(1)}$ is the counting measure on
$\Sigma_K^{(1)}$ associated to $\mu_1$. Write  $A=\{\psi_1,\ldots,
\psi_k\}$, where $\psi_j\in [S]$, $j=1,\ldots, k$. 

Let $n$ be an integer greater than $\max \{k, 4/\varepsilon\}$ and
$\{Y_i\}_{i\in \ZI/n\ZI}$ be a partition of $X\backslash Y$ into $n$
measurable subsets of equal measure $\delta<\varepsilon/4$. Choose
$Z\subset Y$ such that $\mu(Z)=\delta$ and consider partial isomorphisms 
\[
\theta : Z \to Y_0
\]
and 
\[
\varphi_{i} :Y_{i}\to Y_{i+1}, \ i \in \ZI/n\ZI, 
\]
whose graphs are included in $R$ and such that the automorphism $\varphi=\amalg_{i\in \ZI/n\ZI} \varphi_i$ induces an action of $\ZI/n\ZI$ on $X\backslash Y$.  Then 
\[
K'=K \cup\{\theta\}\cup \{\varphi\}
\] 
is a graphing of $R$. Denote by $\Sigma_K$ the  Cayley graph of $S$
associated to $K$ and $\Sigma_{K'}$ the  Cayley graph of $R$ associated to
$K'$.   For $j=1,\ldots, k$, consider the automorphism of $X$ defined by 
\[
\psi_j'=\varphi^{j}\amalg\psi_j
\] 
and observe that the graphs of $\psi'_j$, ${j=1,\ldots, k}$, are pairwise
disjoint. Set $A'=\{\psi_j'\}_{j=1,\ldots, k}$.
 Then  for $y\in Y$, one has 
\[
\abs{ (\del_{K'}A')^y} \leq \abs{(\del_KA)^y}+ \abs{\{\psi\in A \,;\; \psi(y)\in
Z\}},
\]
while for $y\in X\backslash Y$, one has $\abs{ (\del_{K'}A')^y} \leq 3$. Thus
 \[
\nu^{(1)}(\del_{K'}A') \leq \mu(Y)\nu_1^{(1)}(\del_KA) +k\delta
+3\mu(X\backslash Y),
\]
 where $\nu^{(1)}$ is the counting measure on $\Sigma^{(1)}$ associated to $\mu$. It follows that 
 \begin{align*}
 h(R)&\leq \frac{\nu^{(1)}(\del_{K'}A') }{ \abs{A'}}\\
 &\leq \mu(Y)(h(S) +\varepsilon/2)+\varepsilon/2\\ 
& \leq \mu(Y)h(S) +\varepsilon.
 \end{align*}
 So $h(R)\leq \mu(Y) h(S)$ as required. 
\end{proof}

\subsection{Ergodic isoperimetric constant of countable groups}\label{ErgodicCheeger}

Let $\G$ be a finitely generated countable group. We define the
\emph{ergodic  isoperimetric constant}  of $\G$ by the expression
\[
h_e(\G)=\inf_{\alpha} h(R_\alpha),
\]
where the infimum is taken over all ergodic essentially free
measure-preserving actions $\alpha$ of $\G$ on a probability space,
$R_\alpha$ is the partition into the orbits of $\alpha$, and
$h(R_\alpha)$ is the uniform isoperimetric constant of $R_\alpha$. The
following proposition gives another proof of the optimal inequality in the
group case (see Section \ref{perco}).

\begin{proposition}\label{h-actions}
Let $\G$ be an infinite countable group and $\alpha$ be an ergodic essentially free
measure-preserving action of $\G$ on a probability space $(X,\mu)$. Let
$R_\alpha$ be the orbit partition of $X$ into the orbits of $\alpha$. Then
$2\beta_1(\G)\leq 2C(\G)-2 \leq h(R_\alpha)\leq h(\G)$. 
\end{proposition}

\begin{proof}
The first inequalities are from Corollary 3.22  in \cite{Gaboriau02} and Theorem
\ref{main}, while the last one follows from the definitions (simply note that
for any Cayley graph $Y$ of $\G$ and any distinct vertices
$\gamma_1,\gamma_2\in \G=Y^{(0)}$, the graphs of $\alpha(\gamma_1^{-1})$ and $\alpha(\gamma_2^{-1})$ are disjoint vertices in the corresponding Cayley graph of $R_\alpha$ because $\alpha$ is  essentially free).
\end{proof}

Thus we have $2\beta_1(\G)\leq 2C(\G)-2\leq h_e(\G)\leq h(\G)$. Examples of groups for which $h_e(\G)<h(\G)$ are given below. 
\begin{corollary}
If a group has cost 1, then its ergodic isoperimetric constant is zero.
\end{corollary}

\begin{proof}
This follows from Proposition \ref{hzero} and from the fact that the infimum over the actions of the group occurring in the definition of the cost is attained \cite[Proposition VI.21]{Gaboriau99}. 
Note that the infimum in \cite{Gaboriau99} is taken over all (not
necessarily ergodic) measure-preserving essentially free actions, but this infimum is attained (and thus is a minimum) for an ergodic action, as one easily sees  by replacing the infinite product measure by an ergodic joining \cite{Glasner} in the proof of Proposition VI.21 in \cite{Gaboriau99}.
\end{proof}

Thus the class of uniformly non-amenable groups appears to be much larger
from the geometric point of view than from the ergodic point of view.
Breuillard and Gelander  proved in \cite{Breuillard} that for an arbitrary
field $K$, any non-amenable and finitely generated subgroup of $\GL_n(K)$
is uniformly non-amenable. This is the case, for instance,  for
$\SL_3(\ZI)$,
while this group has cost 1 (see \cite{Gaboriau99}) and thus
$h_e(\SL_3(\ZI))=0$. More generally, if $\G$ is a lattice in a semi-simple
Lie group of real rank at least 2, then $\G$ has cost 1 by Corollaire VI.30
in \cite{Gaboriau99}. Note that non-uniform lattices have fixed price
\cite{Gaboriau99} and in that case any measure-preserving action gives an
equivalence relation with trivial uniform isoperimetric constant.

\begin{corollary}[Compare \cite{Breuillard}]\label{lattice}
Lattices in a semi-simple Lie group of real rank at least 2 have trivial
ergodic isoperimetric constant.
\end{corollary}

For the case of direct product of groups, one gets the following result.

\begin{corollary}\label{product}
Finitely generated groups that are decomposable as a direct product of two
infinite groups have trivial ergodic isoperimetric constant.
Finitely generated equivalence relations that are decomposable as a direct
product of two equivalence relations with infinite classes have trivial
uniform isoperimetric constant.
\end{corollary}

\begin{proof}
This follows from the fact that the cost of a direct product is 1
(see \cite[Proposition VI.23]{Gaboriau99} and \cite[Proposition 6.21]{KM}).
\end{proof}

\begin{appendix}

\section{Comparison with alternative definitions of uniform amenability}\label{Hyperb}

In this section, we analyze the differences and analogies between the different notions of uniform non-amenability.
 In \cite{Arj}, Arzhantseva, Burillo, Lustig, Reeves, Short and Ventura
 give a definition of \emph{F{\o}lner constants}, which is very close to
 our definition of isoperimetric constants. The only difference in the
 definition is the fact that they consider the inner boundary while we
 consider the edge boundary.
For a finitely generated group $\Gamma$ with generating system $S$ and
$A$ a finite part of the Cayley graph of $\Gamma$, one has
$$\mbox{F{\o}l}_S (\Gamma,A)=\frac{\abs{ \del_S^{int} A}}{ \abs  A} \mbox{
with }  \del_S^{int} A=\{ a \in A\,;\; \exists x \in S\cup S^{-1}, ax \not \in A\} ,$$
while our boundary is $\del_S A=\{ (a, ax)\,;\;  x \in S\cup S^{-1},  a\in A, ax \not \in A\}$ with $(a,ax)$ the edge between the vertices $a$ and $ax$.

Considering sets $A$ without isolated vertices, one  immediately gets that
\begin{equation}\label{compar}
\mbox{F{\o}l}_S (\Gamma,A) \leq h_S(\Gamma,A) \leq (2\abs
S-1)\mbox{F{\o}l}_S (\Gamma,A)\,,
\end{equation}
whence one has the following proposition.
\begin{proposition}
Let $\mbox{F{\o}l}_S (\Gamma)$ be the F{\o}lner constant defined in \cite{Arj}. Then 
$$\mbox{F{\o}l}_S (\Gamma)\leq h_S(\Gamma) \leq (2\abs  S-1)\mbox{F{\o}l}_S(\Gamma).$$
In particular, $\mbox{F{\o}l}(\Gamma) \leq h(\Gamma) $.
\end{proposition}
This means that a uniformly non-amenable group in the sense of \cite{Arj} is uniformly non-amenable in our sense.  
It is unclear whether the converse is true:  there might exist 
groups that are not uniformly non-amenable in the sense of \cite{Arj} but
are uniformly non-amenable in our sense. However, for all known examples of
groups that are not uniformly non-amenable in the sense of \cite{Arj},
there exists a maximal bound for the size of the generating systems used to
reach the infimum, so these groups are also not uniformly non-amenable in our sense. 

Another notion of uniform non-amenability was introduced by Osin in
\cite{Osin}, linked with the Kazhdan estimates for the regular
representation $\lambda$. A group is said to be uniformly non-amenable in
the sense of \cite{Osin} if $\alpha(\Gamma) =\inf_S \alpha(\Gamma,S) >0$, where $S$ runs over all finite generating systems of $\Gamma$ and
$$
\alpha(\Gamma, S)
= \inf\left\{ \max_{x\in S}\norm{\lambda(x) u-u} \,;\ u \in \ell^2(\Gamma),
\norm{u} = 1 \right\}.
$$
(Note that this constant is also presented as a Kazhdan constant for the regular representation in \cite {Arj} and denoted by $K(\Lambda_\Gamma, \Gamma)$.)

Consider now a Cayley graph $Y$ associated with a system of generators $S$,
and let $A$ be a finite subset in $Y$ and $u_A$ be the normalized
characteristic function $u ={\chi_{A^{-1}}\over{\sqrt{ \abs  A}}}$. Then one has, for any $x \in S$
$$\norm{\lambda(x)u_A-u_A}^2={1 \over{ \abs  A}}\sum_{ g \in \Gamma}[
\chi_{A^{-1}} (x^{-1}g) -\chi_{A^{-1}} (g)]^2 \leq {1 \over{ \abs
A}}\sum_{x \in S, g \in \Gamma}[ \chi_{A} (gx) -\chi_{A} (g)]^2  = {\abs{
\del_S A}\over \abs  A} \cdot$$
This implies that
$$
\alpha(\Gamma, S)
\leq \norm{\lambda(x)u_A-u_A} \leq \sqrt{h_S(\Gamma,A)}\,,
$$
whence we have proved the following proposition.
\begin{proposition}
Let $\alpha(\Gamma)$ be the constant introduced by Osin \cite{Osin}. Then one has
$$\alpha(\Gamma) \leq \sqrt{h(\Gamma)}.$$
\end{proposition}
This means that a uniformly non-amenable group in the sense of Osin is
still uniformly non-amenable in our sense, while the converse seems to be an open question \cite{Arj}.
\section{Some properties of isoperimetric constants}\label{propgroup}
In this section, we restate properties showed in \cite{Arj} on F{\o}lner
constants for the isoperimetric constants we introduce. We have omitted the proofs as they are \emph{quasi verbatim} the ones in \cite{Arj}.
First of all, isoperimetric constants are linked to exponential growth.
Recall that the exponential growth rate of a group $\Gamma$ finitely
generated by $S$ is defined as the limit $\displaystyle{\omega_S(\Gamma)=
\lim_{n\to  \infty}\sqrt[n]{\abs{ B_S(n)}}}$ where $B_S(n)$ is the ball of
elements in $\Gamma$ with geodesic distance (as $S$ words) at most $n$.

\begin{proposition}[Isoperimetric constant and exponential growth]\label{growth} \ \\
Let $\Gamma$ be a finitely generated group and $S$ a finite generating system. Then
$$h(\Gamma)\leq h_S (\Gamma)\leq (2\abs  S-1)[1-{1 \over{\omega_S(\Gamma)}}].$$
\end{proposition}
Note that we don't know in general in our case if uniform non-amenability
implies uniform exponential growth. This is the case when we know that the
infimum can be attained for systems of generators with uniform bound on the
cardinalities of these systems.

Concerning the isoperimetric constants for subgroups and quotients, the theorems of \cite{Arj} are exactly the same.
\begin{theorem}[Subgroups]\label{subgroups}
\ \\
Let $\Gamma$ be a finitely generated group and $S$ a finite system of generators.
\begin{enumerate}
\item  Consider the subgroup $\Gamma'$ generated by a subsystem $S' \subset S$. Then $\displaystyle{h_S(\Gamma)\geq h_{S'}(\Gamma') .}$
\item Let $H$ be a subgroup of $\Gamma$ generated by a system $T$ with $m$ elements and such that the length of each element of $T$ as a word in $S$ is at most $L$. Then $\displaystyle{h_S(\Gamma)\geq {h_T(H) \over 1+mL}.}$
\end{enumerate}
\end{theorem}

\begin{theorem}[Quotients]\label{quotients}
Let $\Gamma$ be a finitely generated group and $S$ a finite system of
generators. Denote by $N$ a normal subgroup and by $\pi : \Gamma \to
\Gamma/N$ the natural projection. Then $\displaystyle{h_S(\Gamma)\geq
h_{\pi(S)}(\Gamma/N)}$, whence $\displaystyle{h(\Gamma)\geq h(\Gamma/N)}$.
\end{theorem}

\section{Some bounds on isoperimetric constants}\label{bounds}
In this section, we give an explicit calculation for the free group, as in \cite{Arj},  and derive some bounds for groups by quotient and subgroup theorems.
\begin{proposition}\label{free}
For the free group on $k$ generators, one has
$h(F_k)=2k-2$.
\end{proposition}
\begin{proof}
By Theorem \ref{fsf}, we know that $h(F_k)\geq 2 \beta_1(\Gamma)=2k-2$. On the other hand, by Proposition \ref{growth}, we know that
$$h(F_k)\leq (2k-1)(1-{1\over \omega_S(F_k)})$$
and that $\omega_S(F_k)=2k-1$ for $S$ a system of $k$ free generators of $F_k$.
\end{proof}

From this result and the quotients and subgroup theorems, we can, as in
\cite{Arj}, derive the following for groups.
\begin{proposition}\
Let $\Gamma$ be a finitely generated group that admits a system $S$ of $k$ generators. Then $h(\Gamma)\leq 2k-2$ with equality if only if $\Gamma$ is a free group $F_k$.
\end{proposition}

\bigskip
N.B. The present paper substitutes and extends a short  note by the second
author circulating  during his Ph.D.\ under the same title, where the group case only  (Theorem \ref{CG-groups}) was considered. 
 
\end{appendix}


\begin{thebibliography}{00}

\bibitem{AnanRen}
C. Anantharaman-Delaroche and J. Renault. ``Amenable Groupoids", 
L'Enseignement Math\'{e}matique, Geneva, \textit{2000}.


\bibitem{Arj} Arzhantseva G., Burillo J., Lustig M., Reeves L., Short H.,
Ventura E., ``Uniform non-amenability", Adv. in Math. 197 (2), 499--522,
\textit{2005}.



\bibitem{Atiyah76} Atiyah M. F.,  ``Elliptic operators, discrete groups and von Neumann algebras",  Colloque ``Analyse et Topologie" en l'Honneur de Henri Cartan (Orsay, 1974), pp. 43--72. Asterisque, No. 32-33, Soc. Math. France, Paris, \textit{1976}.

\bibitem{BLPSusf}
Benjamini I., Lyons R., Peres Y., Schramm O.
``Uniform spanning forests",
Ann. Probab. 29, 1--65,
\textit{2001}.


\bibitem{Breuillard} Breuillard E., Gelander T., ``Cheeger constant and algebraic entropy of linear groups", Int. Math. Res. Not., no. 56, 3511--3523, \textit{2005}. 

\bibitem{CG86} Cheeger J., Gromov M., ``$L\sb 2$-cohomology and group cohomology",  Topology 25, no. 2, 189--215, \textit{1986}.

\bibitem{CFW81} Connes A., Feldman J., Weiss B., ``An amenable equivalence relation is generated by a single transformation", Ergod. Th. \& Dynam. Sys., 1, 431-450, \textit{1981}.

\bibitem{FM77} Feldman J., Moore C.,  ``Ergodic equivalence relations, cohomology, and von Neumann algebras. I",  Trans. Amer. Math. Soc., 234(2), 289-324, \textit{1977}. 


\bibitem{Gaboriau99} Gaboriau D., ``Co\^ut des relations d'\'equivalence et des groupes",  Invent. Math.  139,  no. 1, 41--98, \textit{2000}.

\bibitem{Gaboriau02} Gaboriau D., ``Invariants $l\sp 2$ de relations d'\'equivalence et de groupes", Publ. Math. Inst. Hautes \'Etudes Sci.  No. 95, 93--150, \textit{2002}.

\bibitem{GabHD} 
Gaboriau, D. 
``Invariant percolation and harmonic {D}irichlet functions",
Geom. Funct. Anal. 15, 1004--1051,
\textit{2005}.

\bibitem{GabLyons} 
Gaboriau, D., Lyons R.,
``A measurable-group-theoretic solution to von
Neumann's problem", preprint,
\textit{2007}.


\bibitem{Glasner}
 Glasner E.,  ``Ergodic Theory via Joinings", American Mathematical Society,
Providence, Rhode Island, \textit{2003}.


\bibitem{Gromov93} Gromov M., ``Asymptotic invariants of infinite groups",  Geometric group theory, Vol. 2 (Sussex, 1991), 1--295, London Math. Soc. Lecture Note Ser., 182, Cambridge Univ. Press, Cambridge, \textit{1993}. 

\bibitem{Harpe00} de la Harpe P., ``Topics in geometric group theory", Chicago Lectures in Math. Series, \textit{2000}.


\bibitem{KM}
Kechris A.,  Miller B.,
``Topics in orbit equivalence",
Lecture Notes in Mathematics 1852. Berlin: Springer, \textit{2004}. 

\bibitem{Luck02} L\"uck W., ``$L\sp 2$-invariants: theory and applications to geometry and $K$-theory", Ergebnisse der Mathematik und ihrer Grenzgebiete. 3. Folge. A Series of Modern Surveys in Mathematics [Results in Mathematics and Related Areas. 3rd Series. A Series of Modern Surveys in Mathematics], 44. Springer-Verlag, Berlin, \textit{2002}.

\bibitem{Luck} L\"uck W. ``Dimension theory of arbitrary modules over finite von Neumann algebras and $L\sp 2$-Betti numbers, I Foundations",  J. Reine Angew. Math. 495, 135--162, \textit{1998}. 

\bibitem{LPSmsf} 
Lyons R., Peres Y., Schramm O.,
``Minimal spanning forests",
Ann. Probab. 34, 1665--1692,
\textit{2006}.

\bibitem{Lyonsbetti}
Lyons R.,
``Random complexes and $\ell^2$-{B}etti numbers",
in preparation.

\bibitem{Lyonsphase}
Lyons R.,
``Phase transitions on nonamenable graphs", J. Math. Phys.
41, 1099--1126, \textit{2000}.

\bibitem{Osin} Osin D., ``Weakly amenable groups",  Computational and statistical group theory (Las Vegas, NV/Hoboken, NJ, 2001), 105--113, Contemp. Math., 298, Amer. Math. Soc., Providence, RI, \textit{2002}. 

\bibitem{Osin2} Osin D., ``Uniform non-amenability of free Burnside
groups", Arch. Math. (Basel) 88, no. 5, 403--412, \textit{2007}.

\bibitem{Pansu96} Pansu P., ``Introduction to $L\sp 2$ Betti numbers",  Riemannian geometry (Waterloo, ON, 1993), 53--86,  Fields Inst. Monogr., 4, 
Amer. Math. Soc., Providence, RI, \textit{1996}. 

\bibitem{sc} Pichot M., ``Semi-continuity of the first $\ell^2$-Betti number on the space of finitely generated groups'', Comm. Math. Helv.,  81, No. 3, 643--652, \textit{2006}.

\bibitem{ef} Pichot M., ``Espaces mesur\'es singuliers fortement ergodiques
--- \'Etude m\'etrique-mesur\'ee",  Ann. Inst. Fourier, 57, no. 1, 1--43, \textit{2007}.

\bibitem{ts} Pichot M., ``Sur la th\'eorie spectrale des relations d'\'equivalence mesur\'ees'', 
J. Inst. Math. Jussieu, Vol. 6, no. 03, 453--500, \textit{2007}.

\bibitem{th} Pichot M., ``Quasi-p\'eriodicit\'e et th\'eorie de la mesure",
Ph.D. Thesis. \'Ecole Normale Superieure de Lyon, \textit{2005}.


\end{thebibliography}
\end{document}